\documentclass[11pt]{amsart}

\usepackage{latexsym,rawfonts}
\usepackage{amsfonts,amssymb}
\usepackage{amsmath,amsthm}

\usepackage[plainpages=false]{hyperref}
\usepackage{graphicx}
\usepackage{color}
\usepackage[table]{xcolor}
\usepackage{longtable}

\newcommand{\R}{\mathbb{R}}

\newcommand{\abs}[1]{\left\vert#1\right\vert}

\newcommand{\delbar}{\overline{\nabla}}

\newcommand{\rnnn}{\mathbb R^{n}}

\newcommand{\sn}{ {\mathbb{S}^{n-1}}}
\newcommand{\rn}{\mathbb R}

\newcommand{\psum}{{+_{\negthinspace\kern-2pt p}}\,}
\newcommand{\qsum}[1]{{+_{\negthinspace\kern-2pt #1}}\,}
\newcommand{\dpsum}{{\tilde+_{\negthinspace\kern-1pt p}}\,}
\newcommand{\dqsum}[1]{{\tilde+_{\negthinspace\kern-1pt #1}}\,}
\newcommand{\lsub}[1]{\hskip -1.5pt\lower.5ex\hbox{$_{#1}$}}

\numberwithin{equation}{section}

\newtheorem{theo}{Theorem}[section]

\newtheorem{lem}[theo]{Lemma}

\theoremstyle{definition}

\begin{document}

\title{The chord Gauss curvature flow and its $L_{p}$ chord Minkowski problem}

\author[J. Hu]{Jinrong Hu}
\address{School of Mathematics, Hunan University, Changsha, 410082, Hunan Province, China}
\email{Hu\_jinrong097@163.com}
\author[Y. Huang]{Yong Huang}
\address{School of Mathematics, Hunan University, Changsha, 410082, Hunan Province, China}
\email{huangyong@hnu.edu.cn}
\author[J. Lu]{Jian Lu}
\address{School of Mathematical Sciences, South China Normal University, 510631,  Guangzhou, China}
\email{lj-tshu04@163.com}
\author[S. Wang]{Sinan Wang}
\address{School of Mathematics, Hunan University, Changsha, 410082, Hunan Province, China}
\email{wangsinan@hnu.edu.cn}

\begin{abstract}
In this paper, the $L_{p}$ chord Minkowski problem is concerned. Based on the results showed in \cite{HJ23}, we obtain a new existence result of solutions to this problem in terms of smooth measures by using a nonlocal Gauss curvature flow for $p>-n$ with $p\neq 0$.
 \end{abstract}

\thanks{The research is supported by the National Natural Science Foundation of China (12171144,12231006 to Huang, 12122106 to Lu).}
\maketitle

\baselineskip18pt

\parskip3pt

\section{Introduction}
\label{Sec1}

The classical Brunn-Minkowski theory lies in the core of convex geometry. The geometric measures of convex bodies in the Euclidean space $\rnnn$ and their associated Minkowski problems capture the central positions in the framework of the Brunn-Minkowski theory. As a prelude to this subject, the classical Minkowski problem characterizing the surface area measure, which was proposed and solved by Minkowski himself in \cite{M897,M903}. Since then, the existence, uniqueness and regularity of this problem  have quickly been systematically investigated in a series of paper \cite{A39,A42,CY76,CW06,FJ38,HgLYZ05,L93,LYZ04,N53,Zhu14,Zhu15}. Apart from area measures, the curvature measures introduced by Federer \cite[p. 224]{S14} compose another nonnegligible family of measures in the Brunn-Minkowski theory.

As one generalization of the Brunn-Minkowski theory, the $L_{p}$ Brunn-Minkowski theory, initiated by Firey but which was indeed endowed life by Lutwak \cite{L93,L96} when he studied the $p$-Minkowski combination and defined the $L_{p}$ surface area measure. The $L_p$ Minkowski problem prescribing the $L_{p}$ surface area measure is the fundamental problem, which was first posed and attacked by Lutwak ~\cite{L93}. After that, the researches on this problem have been the breeding ground for many valuable results in \cite{B19,B17,CL17,F62,LW13,Zhu14,Zh15,Zhu15} and their references. The most challenging cases include, the log-Minkowski problem for $p=0$ (see B\"{o}r\"{o}czky-LYZ \cite{BLYZ12} and their references), and the centro-affine Minkowski problem for $p=-n$ (see Chou-Wang \cite{CW06} or Zhu \cite{Zhu15}, etc.).

As another extension of the Brunn-Minkowski theory, the dual Brunn-Minkowski theory, was launched by Lutwak in the 1970s \cite{L75}, but it indeed gained life recently when Huang-LYZ \cite{HLYZ16} discovered a new family of fundamental geometric measures, named as the dual curvature measure. These measures are dual to Federer's curvature measures, which bear a same role as area measures and curvature measures in the Brunn-Minkowski theory. This work leads to the emergence of the dual Minkowski problem prescribing dual curvature measures and provides the sufficient condition on the existence of solution to this problem.  Since then, the dual Brunn-Minkowski theory  have been obtained the rapid and flourishing developments for many valuable outputs and applications, for more reading, such as refer to \cite{BF19,BLYZ19,BLYZ20,HLYZ18,HZ18,JW17,LSW20,Z17,Z18}.

 Very recently, Lutwak-Xi-Yang-Zhang \cite{LXYZ} discovered a new family of geometric measures, deemed as chord measures, which are the the differential of chord integrals coming from integral geometry (For more knowledge on integral geometry, readers are recommended to see books of Santal\'{o} \cite{S04} and Ren \cite{R94}). The chord integral $I_{q}(\Omega)$ of the convex body $\Omega\subset \rnnn$ is defined by
\begin{equation}\label{chord}
I_{q}(\Omega)=\int_{\pounds^{n}}|\Omega \cap \ell|^{q}d\ell, \quad {\rm real}\ q\geq 0,
\end{equation}
where $|\Omega \cap \ell|$ denotes the length of the chord $\Omega \cap \ell$, and the integration is with respect to Haar measure on $\pounds^{n}$ as normalized above. As can be seen in \cite[p. 5]{LXYZ}, the formula \eqref{chord} includes the convex body's surface area and volume as two important cases when the index $q=0$ and $q=1$ or $q=n+1$ respectively.


As showed in \cite{R94,S04}, for $q>1$, from the perspective of analysis, there is an amazing link between the chord integrals and the Riesz potentials of the characteristic function of the convex body $\Omega$,
\begin{equation}\label{q1} I_{q}(\Omega)=\frac{q(q-1)}{n\omega_{n}}\int_{\rnnn}\int_{\rnnn}\frac{\chi_{\Omega}(\vartheta)\chi_{\Omega}(\zeta)}{|\vartheta-\zeta|^{n+1-q}}d\vartheta d\zeta, \quad  q>1.
\end{equation}

Now we turn to talk about  chord measures, which stem from the differential of the chord integrals. As  revealed by \cite[Lemma 5.6]{LXYZ}, for $q>0$, the variational formula of chord integrals is given by
 \begin{equation}\label{IQR}
  \frac{d}{dt}\Big|_{t=0}I_{q}(\Omega+tL)=\int_{\sn}h_{L}(v)dF_{q}(\Omega,v),
   \end{equation}
 where $F_{q}(\Omega,\cdot)$ is the chord measure of $\Omega$, defined by
 \[
  F_{q}(\Omega,\eta)=\frac{2q}{\omega_{n}}\int_{\nu^{-1}_{\Omega}(\eta)}\widetilde{V}_{q-1}(\Omega,z)d \mathcal{H}^{n-1}(z), \quad {\rm Borel}\ \eta\subset \sn,
  \]
  here $\widetilde{V}_{q-1}(\Omega,\cdot)$ is a nonlocal term, more precisely, a Riesz potential defined on the boundary of convex body,  and is also deemed as the $q$-th dual quermassintegral of $\Omega$ (see Sec. \ref{Sec2} for its definition, we omit it here).
  Chord measures shall be served as the fourth major family of geometric measures of convex bodies that together with Aleksandrov-Fenchel-Jessen's area measures, Federer's curvature measure and the dual curvature measures introduced by Huang-LYZ \cite{HLYZ16}.
 Furthermore, for $q\geq0$ and $p\in \R$, Lutwak-Xi-Yang-Zhang \cite{LXYZ} also defined the $L_{p}$ chord measure $F_{p,q}(\Omega,\cdot)$ of $\Omega$ as
 \[
 dF_{p,q}(\Omega,\cdot)=h_{\Omega}(\cdot)^{1-p}dF_{q}(\Omega,\cdot).
  \]

Along above routes, we arrive the entrance in exploring the geometric problems. In \cite{LXYZ}, the authors introduced the $L_{p}$ chord Minkowski problem as below.

{\bf The $L_{p}$ Chord Minkowski Problem.} Suppose $q\geq 0$ and $p\in \R$, given a finite Borel measure $\mu$ on the unit sphere $\sn$, what are the necessary and sufficient conditions on $\mu$ to ensure the existence of a convex body $\Omega\subset \rn$ such that
\begin{equation}\label{chm}
\mu=F_{p,q}(\Omega,\cdot)?
\end{equation}
The authors \cite{XYZ23} gave the existence of the solution of \eqref{chm} by applying variational arguments  for $p>1$ and $0<p<1$ under the symmetric assumption. Very recently, Guo-Xi-Zhao \cite{GXZ23} solved the $L_{p}$ chord Minkowski problem in the case of $0\leq p<1$ without any symmetric assumption, and Li \cite{Li23} treated the discrete $L_{p}$ chord Minkowski problem in the condition of $p<0$ and $q>0$, as for general Borel measure, Li also gave a proof but need $p\in (-n,0)$ and $1<q<n+1$. Note that, in view of \eqref{chm}, this problem becomes the $L_{p}$ Minkowski problem characterized in the case $q=1$ or $q=n+1$, when $p=0$ and $q=1$ or $q=n+1$, it is the classical log-Minkowski problem, which was attacked by B\"{o}r\"{o}czky-LYZ \cite{BLYZ12} in the circumstance that $\mu$ is an even measure, and it becomes to the chord Minkowski problem with $p=0$ and the chord log-Minkowski problem with $p=1$ studied by \cite{LXYZ} from the perspective of variational theory. Very lately, the regularity of the solution to the chord log-Minkowski problem was studied by Hu-Huang-Lu \cite{HJ23} from the point view of the elliptic equation and parabolic equation of theories.

Purely as an aside, when the given measure $\mu$ has a positive density, say $\frac{2q}{\omega_{n}}f$, the solvability of the geometric problem \eqref{chm} amounts to tackling the following equation of
 {M}onge-{A}mp\`ere type involving dual quermassintegrals of function,
 \begin{equation}\label{Mong}
h^{1-p}\det(\nabla^{2}h+hI)\widetilde{V}_{q-1}(h,\overline{\nabla} h)=f,\quad {\rm on} \ \sn,
\end{equation}
where $h:\sn \rightarrow (0,+\infty)$ is the unknown function, $\overline{\nabla} h$ denotes the Euclidean gradient of $h$ in $\rnnn$, and $\nabla^{2}h$ is the spherical Hessian of $h$ with respect to a local orthonormal frame on $\sn$, $I$ is the identity matrix.

Illuminated by the author's work \cite{HJ23}, the purpose of current paper is to attack the solvability of equation \eqref{Mong} with the aid of the geometric flow. The geometric flow involving Gauss curvature was first introduced by Firey \cite{F74}, it described shapes of worn stones. Since then, a lot of variants of Gauss curvature flow have sprung up, for example, Chou-Wang \cite{CW00} made use of a logarithmic Gauss curvature flow to deal with the classical Minkowski problem, Li-Sheng-Wang \cite{LSW20} took advantage of the related Gauss curvature flow to deal with the Aleksandrov problem and dual Minkowski problem, and Chen-Huang-Zhao \cite{CHZ19} gave smooth solutions to the $L_{p}$ dual Minkowski problem and see their references. For our purpose, different from previous Gaussian curvature flows, we are concerned with a family of convex hypersurfaces $\partial \Omega_{t}$ parameterized by smooth map $X(\cdot,t):\sn\rightarrow \rnnn$ satisfying the following flow equation involving a nonlocal term bearing on the regularity of the boundary of hypersurface,
\begin{equation}\label{GC}
\left\{
\begin{array}{lr}
\frac{\partial X(x,t)}{\partial t}=-f\kappa (X\cdot v)^{p} \frac{\eta(t)}{\int_{\sn}fh^{p}dx}\frac{1}{\widetilde{V}_{q-1}(\Omega_{t},\overline{\nabla} (X\cdot v))} v+X(x,t), \\
X(x,0)=X_{0}(x),
\end{array}\right.
\end{equation}
where $\eta(t)= \frac{\omega_{n}(n+q-1)I_{q}(\Omega_{t})}{2q}$, $\kappa={\rm det}(\nabla^{2}h+hI)^{-1}$ is the Gauss curvature of the hypersurface $\partial \Omega_{t}$ at $X(\cdot,t)$, and $v=x$ is the unit outer normal vector of $\partial\Omega_{t}$ at $X(\cdot,t)$. It should be again remarked that the flow \eqref{GC} was first introduced by Hu-Huang-Lu \cite{HJ23} to get the regularity results of the chord log-Minkowski problem in the context of $p=0$.

We are in the place to elaborate that the main aim is to verify that the solution $h(x,t)$ of the flow \eqref{GC} exists for all time $t\in(0,+\infty)$, and it subconverges to a function
$h(x)$  as $t$ goes to infinity, which is $\left[\frac{\omega_{n}(n+q-1)I_{q}(\Omega)}{2q\int_{\sn}fh^{p}{dx}}\right]^{\frac{1}{n+q-1-p}}$ times the solution of \eqref{Mong}.
As expounded by \cite{HJ23}, the biggest obstacle in the solving process spring up brought by the Riesz potential $\widetilde{V}_{q-1}$ with respect to boundary point, which is more delicate than that  with respect to interior point.  Definitely, the authors \cite{HJ23} put forward a new technique to deal with the above troublesome and obtained the smooth estimation of the Riesz potential on the boundary. Based on their work, the buck of this paper is to derive the following result.
\begin{theo}\label{main}
Suppose $\{p>0, q>3\}\cup \{-n<p<0, 3<q<n+1\}$.  Let $\Omega_{0}$ be a smooth, origin-symmetric and strictly convex body in $\rnnn$. Let $f$ be a positive, even and smooth function on $\sn$. Then
there exists a smooth, origin-symmetric and strictly convex solution $\Omega_{t}$ to
flow \eqref{GC} for all $t>0$   and  its subsequence in $C^{\infty}$ converges to a smooth, origin
symmetric, strictly convex solution to equation \eqref{Mong}.
\end{theo}

The organization of this paper is as follows. In Sec. \ref{Sec2}, we make some preliminaries on convex body and integral geometry. In Sec. \ref{Sec3}, we establish the related nonlocal Gauss curvature flow and the functional. In Sec. \ref{Sec4}, we derive $C^{0}$, $C^{1}$ estimates.  In Sec. \ref{Sec5}, the $C^{2}$ estimate is established. In Sec. \ref{Sec6}, we complete the main theorem.

\section{Preliminaries}
 \label{Sec2}
We will work in the $n$-Euclidean space ${\rnnn}$. Denote by ${\sn}$ the unit sphere. A convex body is a compact convex set of ${\rnnn}$ with non-empty interior. $\omega_{n}$ is the volume of unit ball in $\rnnn$. For $y,z\in {\rnnn}$, $y\cdot z$ denotes the standard inner product.   For quick and standard references on convex body, please see the monographs of Gardner \cite{G06} and Schneider \cite{S14}. In what follows, we shall give some basics on Riesz potentials and chord integrals, for more details, please refer to \cite{LXYZ}.

{\bf Radial and Support function.}  For a convex body $\Omega$ in $\rnnn$ and $z\in \rnnn$, set
\[
S_{z}=S_{z}(\Omega)=\{u\in \sn :\Omega\cap(z+\R u)\neq \emptyset\}.
\]
The (extended) radial function $\rho_{\Omega,z}(x):\rnnn \backslash \{o\}\rightarrow \R$ regarding to $z\in \rnnn$, is defined by
\begin{equation*}\label{p1}
\rho_{\Omega,z}(y)=\max\{\lambda\in\R:\lambda y\in \Omega-z\},\quad y\in \rnnn \backslash \{0\}.
\end{equation*}
By above definition, it is clear to observe that
\begin{equation}\label{p2}
\rho_{\Omega,y+z}=\rho_{\Omega-y,z}=\rho_{\Omega-y-z}, \quad y,z\in \rnnn.
\end{equation}
 On the other hand, the (extended) support function $h_{\Omega,z}:\rnnn\rightarrow \R$ of $\Omega$ with respect to $z$, is defined by
\begin{equation*}\label{h1}
h_{\Omega,z}(x)=\max \{x\cdot y:y \in \Omega-z\}, \quad x\in \rnnn.
\end{equation*}
 Obviously,
\begin{equation}\label{h2}
h_{\Omega,y+z}(x)=h_{\Omega-y,z}=h_{\Omega}(x)-(y+z)\cdot x, \quad x,y,z\in \rnnn.
\end{equation}
For the support function and radial function of $\Omega$ with respect to the origin, we shall write $h_{\Omega}$, $\rho_{\Omega}$ rather than $h_{\Omega,o}$, $\rho_{\Omega,o}$.

The Hausdorff metric $\mathcal{D}(\Omega,\widetilde{\Omega})$ between two convex bodies $\Omega$ and $\widetilde{\Omega}$ in ${\rnnn}$, is expressed as
\[
\mathcal{D}(\Omega,\widetilde{\Omega})=\max\left\{|h_{\Omega}(x)-h_{\widetilde{\Omega}}(x)|:x\in {\sn}\right\}.
\]
Let $\{\Omega_{j}\}$ be a sequence of convex bodies in ${\rnnn}$. For a convex body $\Omega_{0}$ in ${\rnnn}$, if $\mathcal{D}(\Omega_{j},\Omega_{0})\rightarrow 0$, then $\Omega_{j} \rightarrow \Omega_{0}$, or equivalently, $\max_{x\in{\sn}}|h_{\Omega_{j}}(x)-h_{\Omega_{0}}(x)|\rightarrow0$, as $j\rightarrow \infty$.

The map $\nu_{\Omega}:\partial \Omega\rightarrow {\sn}$ denotes the Gauss map of $\partial\Omega$, which is related to an element of $\partial \Omega$ with its outer unit normal. Meanwhile, for $\omega\subset {\sn}$, the inverse Gauss map is expressed as
\begin{equation*}
\nu^{-1}_{\Omega}(\omega)=\{X\in \partial \Omega: \nu_{\Omega}(X) {\rm  \ is \ defined \ and }\ \nu_{\Omega}(X)\in \omega\}.
\end{equation*}
For simplicity in the subsequence, we abbreviate $\nu^{-1}_{\Omega}$ as $F$. In particular, for a convex body $\Omega$ being of class $C^{2}_{+}$, i.e., its boundary is of class $C^{2}$ and of positive Gauss curvature, the support function of $\Omega$ can be written as
\begin{equation}\label{hhom}
h_{\Omega}(x)=x\cdot F(x)=\nu_{\Omega}(X)\cdot X, \ {\rm where} \ x\in {\sn}, \ \nu_{\Omega}(X)=x \ {\rm and} \ X\in \partial \Omega.
\end{equation}
 Let $\{e_{1},e_{2},\ldots, e_{n-1}\}$ be a local orthonormal frame on ${\sn}$, $h_{i}$ be the first order covariant derivatives of $h(\Omega,\cdot)$ on ${\sn}$ with respect to the frame. Differentiating \eqref{hhom} with respect to $e_{i}$ , we get
\[
h_{i}=e_{i}\cdot F(x)+x\cdot F_{i}(x).
\]
Since $F_{i}$ is tangent to $ \partial \Omega$ at $F(x)$, we obtain
\begin{equation}\label{Fi}
h_{i}=e_{i}\cdot F(x).
\end{equation}
Combining \eqref{hhom} and \eqref{Fi}, we have (see also \cite[p. 97]{U91})
\begin{equation}\label{Fdef}
F(x)=\sum_{i} h_{i}e_{i}+h_{\Omega}(x)x=\nabla h_{\Omega}(x)+h_{\Omega}(x)x.
\end{equation}
 On the other hand, since we can extend $h_{\Omega}(x)$ to $\rnnn$ as a 1-homogeneous function $h_{\Omega}(\cdot)$, restricting the gradient of $h_{\Omega}(\cdot)$ on $\sn$, it yields that (see for example \cite[p. 14-16]{Guan}, \cite{CY76})
\begin{equation}\label{hf}
\overline{\nabla} h_{\Omega}(x)=F(x), \ \forall x\in{\sn},
\end{equation}
where $\overline{\nabla}$ is the gradient operator in $\rnnn$. Let $h_{ij}$ be the second order covariant derivatives of $h_{\Omega}(\cdot)$ on ${\sn}$ with respect to the local frame. Then, applying \eqref{Fdef} and \eqref{hf}, we have (see, e.g., \cite[p. 382]{J91})
\begin{equation}\label{hgra}
\overline{\nabla }h_{\Omega}(x)=\sum_{i}h_{i}e_{i}+hx, \quad F_{i}(x)=\sum_{j}(h_{ij}+h\delta_{ij})e_{j},\quad F_{ij}(x)=\sum_{k}(h_{ijk}+h_{k}\delta_{ij})e_{k}-(h_{ij}+h\delta_{ij})x.
\end{equation}
Then, the Gauss curvature of $\partial\Omega$, $\kappa$, i.e., the Jacobian determinant of Gauss map of $\partial\Omega$, is revealed as
\begin{equation*}
\kappa=\frac{1}{\det(h_{ij}+h\delta_{ij})}.
\end{equation*}

{\bf Riesz potentials and chord integrals.} As presented in the introduction, for a convex body $\Omega$ in $\rnnn$, as a production of \eqref{chord}, the chord integral $I_{q}(\Omega)$ shall be shown, by $X$-rays (see \cite{G06}), as
\begin{equation}\label{I1}
I_{q}(\Omega)=\frac{1}{n\omega_{n}}\int_{\sn}\int_{\Omega|u^{\bot}}X_{\Omega}(\vartheta,u)^{q}d\vartheta du,\quad q> -1,
\end{equation}
where $\Omega|u^{\bot}$ is the image of the orthogonal projection of $\Omega$ onto $u^{\bot}$, and the parallel $X$-ray of $\Omega$ is defined as
\begin{equation*}\label{}
X_{\Omega}(z,u)=|\Omega \cap (z+\rn u)|, \quad z \in \rnnn, \ u\in \sn.
\end{equation*}
Note that the relation between the $X$-ray function and the radial function is revealed as
\begin{equation}\label{xray}
X_{\Omega}(\vartheta,u)=\rho_{\Omega,z}(u)+\rho_{\Omega,z}(-u), \ {\rm  when} \ \Omega \cap (\vartheta+\rn u)=\Omega \cap (z+\rn u)\neq\emptyset.
\end{equation}
Obviously, in the light of \eqref{xray}, when $z\in \partial \Omega$, then either $\rho_{\Omega,z}(u)=0$ or $\rho_{\Omega,z}(-u)=0$ for almost all $u\in \sn$. So
\[
X_{\Omega}(z,u)=\rho_{\Omega,z}(u),\quad {\rm or} \quad X_{\Omega}(z,u)=\rho_{\Omega,z}(-u),\quad z\in \partial \Omega,
\]
for almost all $u\in \sn$.

An important property of $I_{q}$ is that it is homogeneous of degree $(n+q-1)$ for $q>-1$(see \cite[Lemma 3.2]{LXYZ}). Furthermore, it should be noticed that for $q>1$, the relationship between chord integrals and the Riesz potentials of characteristic functions of convex bodies is give by (see \cite{R94,S04})
\begin{equation}\label{inte}
I_{q}(\Omega)=\frac{q(q-1)}{n\omega_{n}}\int_{\rnnn}\int_{\rnnn}\frac{\chi_{\Omega}(\vartheta)\chi_{\Omega}(\zeta)}{|\vartheta-\zeta|^{n+1-q}}d\vartheta d\zeta, \quad {\rm real } \ q>1,
\end{equation}
where $\chi_{\Omega}(\cdot)$ is the characteristic function of $\Omega$.

{\bf Dual Quermassintegrals.}
For $q\in \R$,  define the $q$-th dual quermassintegral $\widetilde{V}_{q}(\Omega,z)$ of $\Omega$ with respect to $z\in \Omega$, by
\begin{equation}\label{VQ}
\widetilde{V}_{q}(\Omega,z)=\widetilde{V}^{+}_{q}(\Omega,z).
\end{equation}
On the other hand, for $q>0$, and $z\notin \Omega$, define $\widetilde{V}_{q}(\Omega,z)$ by
\[
\widetilde{V}_{q}(\Omega,z)=\widetilde{V}^{+}_{q}(\Omega,z)-\widetilde{V}^{-}_{q}(\Omega,z).
\]
Here
\[
\widetilde{V}^{+}_{q}(\Omega,z)=\frac{1}{n}\int_{S^{+}_{z}}\rho_{\Omega,z}(u)^{q}du,\quad \widetilde{V}^{-}_{q}(\Omega,z)=\frac{1}{n}\int_{S^{-}_{z}}|\rho_{\Omega,z}(u)|^{q}du,
\]
where
\[
S^{+}_{z}=\{u\in \sn: \ \rho_{\Omega,z}(u)>0\},\quad S^{-}_{z}=\{u\in \sn: \ \rho_{\Omega,z}(u)<0\}.
\]
Since $S^{+}_{z}=\sn$ when $z\in {\rm int} \ \Omega$ and $\rho_{\Omega,z}(u)=0$ when $z\in \partial \Omega$ and $u$ is in the interior of $\sn \backslash S^{+}_{z}$, we obtain
\[
\widetilde{V}_{q}(\Omega,z)=\frac{1}{n}\int_{\sn}\rho_{\Omega,z}(u)^{q}du,
\]
when either $q\in \R$ and $z\in {\rm int} \ \Omega$, or $q>0$ and $z\in \partial \Omega$.

 We are in the place to get the Riesz potential form of the dual quermassintegrals $\widetilde{V}_{q}(\cdot,z)$ with respect to $z\in \partial \Omega$. More precisely, for $q>0$, via the transformation of the polar coordinate, one sees
\begin{equation}\label{cj*}
\widetilde{V}_{q}(\Omega,z)=
\frac{q}{n}\int_{\Omega}\frac{1}{|y-z|^{n-q}}dy.
\end{equation}

{\bf Chord measures and $L_{p}$ chord measures.} As shown in \cite{LXYZ}, the chord measure $F_{q}(\Omega,\cdot)$ is given by
\[
F_{q}(\Omega,\eta)=\frac{2q}{\omega_{n}}\int_{\nu^{-1}_{\Omega}(\eta)}\widetilde{V}_{q-1}(\Omega,z)d \mathcal{H}^{n-1}(z), \quad {\rm Borel}\ \eta\subset \sn.
\]
For each $p\in \R$, the $L_{p}$ chord measure $F_{p,q}(\Omega,\cdot)$ is defined as follows:
 \[
dF_{p,q}(\Omega,\eta)=h_{\Omega}(\cdot)^{1-p}dF_{q}(\Omega,\eta), \quad {\rm Borel}\ \eta\subset \sn.
\]

It should be stressed that, as given in \cite[Lemma 3.3]{LXYZ}, for $q>0$, the expression of $I_{q}(\Omega)$ can also be written as
\begin{equation}\label{IQ}
I_{q}(\Omega)=\frac{q}{\omega_{n}}\int_{\Omega}\widetilde{V}_{q-1}(\Omega,z)dz.
\end{equation}


\section{The flow equation and related functional}
\label{Sec3}
As we presented in the introduction, if a given finite Borel measure has a positive density $\frac{2q}{\omega_{n}}f(x)$ on the unit sphere $\sn$, the existence of the $L_{p}$ chord Minkowski problem  is equivalent to solving the following new {M}onge-{A}mp\`ere equation,
\begin{equation}\label{Exil**}
h^{1-p}\det(h_{ij}+h\delta_{ij})\widetilde{V}_{q-1}(h,\overline{\nabla } h)=f(x), \quad  \forall x \in {\sn}.
\end{equation}
To derive the solvability of \eqref{Exil**}, we capitalize on the related flow equation. Suppose $\Omega_{0}$ is a smooth, origin-symmetric and strictly convex body in $\rnnn$, as presented above, we focus on a family of convex hypersurfaces $\partial \Omega_{t}$ parameterized by smooth map $X(\cdot,t):\sn\rightarrow \rnnn$ satisfying the following flow equation,
\begin{equation}\label{GCF2}
\left\{
\begin{array}{lr}
\frac{\partial X(x,t)}{\partial t}=-f\kappa(X\cdot v)^{p} \frac{\eta(t)}{\int_{\sn}fh^{p}dx}\frac{1}{\widetilde{V}_{q-1}(\Omega_{t},\overline{\nabla } (X\cdot v))} v+X(x,t), \\
X(x,0)=X_{0}(x),
\end{array}\right.
\end{equation}
where $\eta(t)=\frac{\omega_{n}(n+q-1)I_{q}(\Omega_{t})}{2q}$. Multiplying both sides of \eqref{GCF2} by $v$, with the aid of the definition of support function, the flow equation associated with the support function $h(x,t)$ of $\Omega_{t}$ is given by
\begin{equation}\label{GCF22}
\left\{
\begin{array}{lr}
\frac{\partial h(x,t)}{\partial t}=-f\kappa h^{p} \frac{\eta(t)}{\int_{\sn}fh^{p}dx}\frac{1}{\widetilde{V}_{q-1}(\Omega_{t},\overline{\nabla} h)} +h(x,t), \\
h(x,0)=h_{0}(x).
\end{array}\right.
\end{equation}

Remark that the short time existence for \eqref{GCF22} follows along similar line as \cite[Lemma 4.1]{HJ23}.

Now, establishing the functional associated with the flow as
\begin{equation}\label{func2}
J(\Omega_{t})=-\frac{1}{n+q-1}\log I_{q}(\Omega_{t})+\frac{1}{p}\log \int_{\sn}f(x)h(x,t)^{p}dx.
\end{equation}
Under the flow, we have the monotonicity of $I_{q}(\Omega_{t})$ and $J(\Omega_{t})$, which are showed as follows.

\begin{lem}\label{monJ1}
Assume that $\Omega_{t}$ is a smooth, origin-symmetric and strictly convex solution satisfying the flow \eqref{GCF2} in $\rnnn$. Then, $I_{q}(\Omega_{t})$ is a constant along the flow \eqref{GCF2} for $q>0$.
\end{lem}
\begin{proof}
Observe that \cite[Lemma 5.5]{LXYZ}, based on which, by a direct computation and utilizing \cite[Lemma 4.3]{LXYZ}, then we derive
\begin{equation*}
\begin{split}
\label{}
\frac{d}{dt}I_{q}(\Omega_{t})&=\int_{\sn}\frac{\partial h}{\partial t}\frac{2q}{\omega_{n}}\widetilde{V}_{q-1}(\Omega_{t},\overline{\nabla} h)\frac{1}{\kappa}{dx}\\
&=\int_{\sn}\left[-\frac{f\kappa h^{p} \frac{\omega_{n}}{2q}(n+q-1)I_{q}(\Omega_{t})\frac{1}{\widetilde{V}_{q-1}(\Omega_{t},\overline{\nabla} h)}}{\int_{\sn}fh^{p}dx}+h\right]\frac{2q}{\omega_{n}}\widetilde{V}_{q-1}(\Omega_{t},\overline{\nabla} h)\frac{1}{\kappa}dx\\
&=0.
\end{split}
\end{equation*}
This reveals that lemma \ref{monJ1} holds.
\end{proof}

\begin{lem}\label{monJ2}
Assume that $\Omega_{t}$ is a smooth, origin-symmetric and strictly convex solution satisfying the flow \eqref{GCF2} in $\rnnn$. Then, $J(t)$ is nonincreasing  along the flow \eqref{GCF2}, i.e.,
 \[
 \frac{d}{dt}J(t)\leq 0.
 \]
\end{lem}
\begin{proof}
Taking the derivative of both side of \eqref{func2}  with respect to $t$, and utilizing \eqref{GCF22}, there is
\begin{equation*}
\begin{split}
\label{}
\frac{d}{dt}J(\Omega_{t})&=-\int_{\sn}\frac{\frac{\partial h}{\partial t}h\widetilde{V}_{q-1}(\Omega_{t},\overline{\nabla} h)\frac{1}{\kappa}\frac{2q}{\omega_{n}}}{(n+q-1)I_{q}(\Omega_{t})h}dx+\frac{\int_{\sn}f(x)h^{p-1}\frac{\partial h}{\partial t}dx}{\int_{\sn}f(x)h^{p}dx}\\
&=\int_{\sn}\frac{\partial h}{\partial t}\left[\frac{-h}{\kappa h\frac{\omega_{n}}{2q}(n+q-1)I_{q}(\Omega_{t})\frac{1}{\widetilde{V}_{q-1}(\Omega_{t},\overline{\nabla} h)}}+\frac{fh^{p-1}}{\int_{\sn}fh^{p}dx}\right]dx\\
&=\int_{\sn}\frac{\partial h}{\partial t}\left[\frac{-h+\frac{\kappa h^{p} \frac{\omega_{n}}{2q}(n+q-1)I_{q}(\Omega_{t})\frac{1}{\widetilde{V}_{q-1}(\Omega_{t},\overline{\nabla} h)}f}{\int_{\sn}fh^{p}dx}}
{\kappa h\frac{\omega_{n}}{2q}(n+q-1)I_{q}(\Omega_{t})\frac{1}{\widetilde{V}_{q-1}(\Omega_{t},\overline{\nabla } h)}}\right]dx\\
&=-\int_{\sn}\frac{[-h+\frac{\kappa h^{p} \frac{\omega_{n}}{2q}(n+q-1)I_{q}(\Omega_{t})\frac{1}{\widetilde{V}_{q-1}(\Omega_{t},\overline{\nabla} h)}f}{\int_{\sn}fh^{p}dx}]^{2}}{\kappa h\frac{\omega_{n}}{2q}(n+q-1)I_{q}(\Omega_{t})\frac{1}{\widetilde{V}_{q-1}(\Omega_{t},\overline{\nabla} h)}}dx\\
&\leq 0.
\end{split}
\end{equation*}
Hence, the proof is completed.
\end{proof}

\section{$C^{0}$, $C^{1}$ estimates}
\label{Sec4}
In this section, the $C^{0}$, $C^{1}$ estimates of the solution to flow \eqref{GCF2} are established. First, the $C^{0}$ estimate is shown as below.

\begin{lem}\label{C0}
 Suppose $\{p>0, q>0\}\cup \{-n<p<0, 1<q<n+1\}$. Let $f$ be an even, smooth and positive function on $\sn$, and $\Omega_{t}$ be a smooth, origin-symmetric and strictly convex solution satisfying the flow \eqref{GCF2}. Then
\begin{equation}\label{C0*}
\frac{1}{C}\leq h(x,t)\leq C, \ \forall (x,t)\in {\sn}\times (0,+\infty),
\end{equation}
and
\begin{equation}\label{C00*}
\frac{1}{C}\leq \rho(u,t)\leq C, \ \forall (u,t)\in {\sn}\times (0,+\infty).
\end{equation}
Here $h(x,t)$ and $\rho(u,t)$ are the support function and the radial function of $\Omega_{t}$.
\end{lem}

\begin{proof}
Due to $\rho(u,t)u=\nabla h(x,t)+h(x,t)x$. Clearly, we have
\begin{equation}\label{p1}
\min_{\sn} h(x,t)\leq \rho(u,t)\leq \max_{\sn}h(x,t).
\end{equation}
This implies that the estimate \eqref{C0} is tantamount to the estimate \eqref{C00*}. So we only need to establish \eqref{C0} or \eqref{C00*}.

Using the monotonicity of $J(\Omega_{t})$ as revealed in lemma \ref{monJ2}, we derive
\begin{equation}
\begin{split}
\label{JD}
J(\Omega_{0})+\frac{1}{n+q-1}\log I_{q}(\Omega_{0})\geq \frac{1}{p}\log\int_{\sn}f(x)h(x,t)^{p}dx.
\end{split}
\end{equation}
Let $\rho_{max}(t)=\max_{\sn}\rho(\cdot,t)$. By a rotation of coordinate, we suppose that $\rho_{max}(t)=\rho(e_{1},t)$. Since $\Omega_{t}$ is origin symmetric, by the definition of $h(x,t)$, one sees $h(x,t)\geq \rho_{max}(t)|x\cdot e_{1}|$ for $\forall x\in \sn$.

Next we analyse two cases to get the upper bound of $h(x,t)$.

Case 1: in the case $p>0$, together with \eqref{JD}, we get
\begin{equation*}
\begin{split}
\label{}
I_{q}(\Omega_{0})^{\frac{p}{n+q-1}}e^{pJ(\Omega_{0})}\geq\int_{\sn}h(x,t)^{p}f{dx}&\geq \int_{\sn}\rho_{max}(t)^{p}|x\cdot e_{1}|^{p}f{dx}\\
&\geq |\min_{\sn} f|\rho_{max}(t)^{p}\int_{\sn}|x\cdot e_{1}|^{p} {dx}\\
&\geq|\min_{\sn} f|\rho_{max}(t)^{p}c_{0},
\end{split}
\end{equation*}
hence,
\begin{equation}\label{max1}
\rho_{max}(t)\leq \left( \frac{I_{q}(\Omega_{0})^{\frac{p}{n+q-1}}e^{pJ({\Omega_{0}})}}{c_{0}|\min f|}\right)^{\frac{1}{p}}\leq C,
\end{equation}
for some $C>0$, independent of $t$. This implies the upper bound for \eqref{C0*}.

Case 2: in the case $p<0$, applying \eqref{JD} again, there exists a positive constant $C^{*}$, independent of $t$, such that
\begin{equation}
\begin{split}
\label{xip}
C^{*}\leq I_{q}(\Omega_{0})^{\frac{p}{n+q-1}}e^{pJ(\Omega_{0})}\leq\int_{\sn}h(x,t)^{p}f(x){dx}.
\end{split}
\end{equation}
For $q>1$, given by \eqref{IQ}, we know
\begin{equation}\label{C1*}
I_{q}(\Omega_{t})=\frac{q}{n\omega_{n}}\int_{\Omega_{t}}\int_{\sn}\rho^{q-1}_{\Omega_{t},z}(u)dudz.
\end{equation}
As $1<q< n+1$, by means of  H\"{o}lder inequality, we obtain
\begin{equation}
\begin{split}
\label{Lo}
I_{q}(\Omega_{0})=I_{q}(\Omega_{t})&\leq \frac{q}{n\omega_{n}} \left(\int_{\Omega_{t}}\int_{\sn}\rho^{n}_{\Omega_{t},z}(u)du dz\right)^{\frac{q-1}{n}}\left(\int_{\Omega_{t}}\int_{\sn}dudz\right)^{1-\frac{q-1}{n}}\\
&\leq CVol(\Omega_{t})^{\frac{n+q-1}{n}},
\end{split}
\end{equation}
which gives that
\begin{equation}\label{vcl}
Vol(\Omega_{t})\geq \widetilde{C}
\end{equation}
for a positive constant $C$, independent of $t$.

 To proceed further, corresponding to case 2, we take by contradiction in the spirit of \cite[p.58]{CW06}. Assume that there is a sequence of convex bodies $\{\Omega_{t_{j}}\}$ satisfying  $ \max_{\sn} h(x,t_{j})\rightarrow \infty$ as $j\rightarrow \infty$. Since $\Omega_{t_{j}}$ is origin symmetric, by John's lemma, there is $\frac{1}{n}E_{t_{j}}\subseteq \Omega_{t_{j}}\subseteq E_{t_{j}}$, where $E_{t_{j}}$ denotes the minimum ellipsoid of $\Omega_{t_{j}}$. Therefore, $\frac{1}{n}h_{E_{t_{j}}}\leq h(x,t_{j})\leq h_{E_{t_{j}}}$. For any fixed small constant $\delta>0$, we decompose $\sn$ into three sets as follows:
\[
S_{1}:=\sn\cap \{h_{E_{t_{j}}}<\delta\}, \ S_{2}:=\sn\cap \{\delta\leq h_{E_{t_{j}}}\leq\frac{1}{\delta}\}, \ S_{3}:=\sn\cap \{h_{E_{t_{j}}}>\frac{1}{\delta}\}.
\]
On the one hand, using \eqref{xip}, for $p<0$, we have
\begin{equation}\label{p1}
C^{*}\leq \int_{\sn}h(x,t_{j})^{p}f(x)dx \leq \int_{\sn}(\frac{h_{E_{t_{j}}}}{n})^{p}f(x)dx.
\end{equation}
On the other hand, since $I_{q}(\Omega_{0}) \leq I_{q}(E_{t_{j}})\leq c(n)^{-(n+q-1)}I_{q}(\Omega_{0})$,  for $-n<p<0$, by Blaschke-Santal\'{o} inequality, H\"{o}lder inequality and \eqref{vcl},  as $\max_{\sn}h(x,t_{j})\rightarrow \infty$ for $j\rightarrow \infty$,
\begin{equation}
\begin{split}
\label{Up3}
\int_{S_{1}}(\frac{h_{E_{t_{j}}}}{n})^{p}f(x)dx &\leq c_{0}\left( \int_{S_{1}}(\frac{h_{E_{t_{j}}}}{n})^{-n} dx \right)^{-\frac{p}{n}}|S_{1}|^{\frac{p+n}{n}}\\
&\leq c_{1}|S_{1}|^{\frac{p+n}{n}}\rightarrow 0.
\end{split}
\end{equation}
Similarly, we obtain
\begin{equation}
\begin{split}
\label{Up3}
\int_{S_{2}}(\frac{h_{E_{t_{j}}}}{n})^{p}f(x)dx\rightarrow 0.
\end{split}
\end{equation}
Also, there is
\begin{equation}\label{}
\int_{S_{3}}(\frac{h_{E_{t_{j}}}}{n})^{p}f(x)dx\leq c_{2} \int_{S_{3}}(\frac{1}{n\delta})^{p}dx=c_{2}(\frac{1}{n\delta})^{p}|S_{3}|\leq c_{3}\delta^{-p}.
\end{equation}
Thus, for any $\delta>0$, employing \eqref{p1}, we derive
\[
C^{*}\leq o(1)+c_{2}\delta^{-p}.
\]
Taking $\delta\rightarrow 0$, we reach a contradiction. So, we conclude that $h(x,t)$ is bounded above uniformly.

To give the lower bound of $h(x,t)$, we take by contradiction technique. Let $\{t_{k}\}\subset [0,+\infty)$ be a sequence such that $h(x,t_{k})$ is not uniformly bounded away from 0, i.e.,
\[
\min_{\sn}h(\cdot, t_{k})\rightarrow 0 \quad as \quad t_k\rightarrow+\infty.
\]
Along another side, making use of the upper bound of \eqref{C0*}, by the Blaschke selection theorem (see \cite{S14}), there is a sequence in $\{\Omega_{t_{k}}\}$, which is still denoted by $\{\Omega_{t_{k}}\}$, such that
\[
\Omega_{t_{k}}\rightarrow \widetilde{\Omega}\quad as \quad t_k\rightarrow+\infty.
\]
Since $\Omega_{t_{k}}$ is an origin-symmetric convex body, $\widetilde{\Omega}$ is also origin-symmetric. And we have
\[
\min_{\sn}h_{\widetilde{\Omega}}(\cdot)=\lim_{k\rightarrow+\infty}\min_{\sn} h_{\Omega_{t_{k}}}(\cdot)=0.
\]
This implies that $\widetilde{\Omega}$ is contained in a lower-dimensional subspace in $\rnnn$. This fact in conjunction with the definition of $I_{q}$ defined in \eqref{I1} for $q>0$ imply that $I_{q}(\widetilde{\Omega})=0$. However, by virtue of lemma \ref{monJ1} and the continuity of $I_{q}$, one sees $I_{q}(\widetilde{\Omega})=I_{q}(\Omega_{0})\neq 0$, which is a contradiction.  It follows that $h(x,t)$ has a uniform lower bound. Hence, we complete the proof of lemma \ref{C0}.
\end{proof}

The $C^{1}$ estimate naturally follows by applying above $C^{0}$ estimate.
\begin{lem}
Suppose $\{p>0, q>0\}\cup \{-n<p<0, 1<q<n+1\}$. Let $f$ be an even, smooth and positive function on $\sn$, and $\Omega_{t}$ be a smooth, origin-symmetric and strictly convex solution satisfying the flow \eqref{GCF2}. Then
\begin{equation}\label{C1*}
|\nabla h(x,t)|\leq C, \ \forall (x,t)\in {\sn}\times (0,+\infty),
\end{equation}
for some $C> 0$, independent of $t$.
\end{lem}
\begin{proof}
Let $u$ and $x$ be related by $\rho(u,t)u=\nabla h(x,t)+h(x,t)x$, we have
\begin{equation*}
 h=\frac{\rho^{2}}{\sqrt{|\nabla \rho|^{2}+\rho^{2}}},\quad \rho^{2}=h^{2}+|\nabla h|^{2}.
\end{equation*}
The above facts together with Lemma \ref{C0} illustrate the desired result.
\end{proof}

\section{$C^{2}$ estimate}
\label{Sec5}
Our purpose is to obtain the $C^{2}$ estimate on the solution of the flow \eqref{GCF2} in this section.

In what follows, we shall give the upper and lower bounds of principal curvature of $\partial \Omega_{t}$ based on above preparations.

\begin{theo}\label{MMM}
Suppose $\{p>0, q>3\}\cup \{-n<p<0, 3<q<n+1\}$. Let $f$ be an even, smooth and positive function on $\sn$, and $\Omega_{t}$ be a smooth, origin-symmetric and strictly convex solution satisfying the flow \eqref{GCF2}. Then
\begin{equation}\label{PUL}
\frac{1}{C}\leq \kappa_{i}\leq C, \ \forall(x,t)\in {\sn}\times (0,+\infty)
\end{equation}
for some positive constants $C$, independent of $t$.
\end{theo}
\begin{proof}
First, we shall show the upper bound of Gauss curvature $\kappa$. It is essential to construct the following auxiliary function,
\begin{equation}\label{AF1}
Q(x,t)=\frac{fh^{p}\kappa\eta(t)\frac{1}{\widetilde{V}_{q-1}(\Omega_{t},\overline{\nabla}h)\int_{\sn}fh^{p}dx}-h(x,t)}{h-\varepsilon_{0}}=\frac{-h_{t}}{h-\varepsilon_{0}},
\end{equation}
where $\eta(t)=\frac{\omega_{n}}{2q}(n+q-1)I_{q}(\Omega_{t})$,  $h_{t}=\frac{d h(x,t)}{dt}$ and
\begin{equation*}\label{AF2}
\varepsilon_{0}=\frac{1}{2}\min_{{\sn}\times (0,+\infty)}h(x,t)>0.
\end{equation*}
For any fixed $t\in (0,\infty)$, assume that the maximum of $Q(x,t)$ is achieved at $x_{0}$. Choosing a local orthonormal frame such that $\{b_{ij}\}$ is diagonal at $x_{0}$ with $b_{ij}=h_{ij}+h\delta_{ij}$. Thus, we obtain that at $x_{0}$,
\begin{equation}\label{Up1}
0=\nabla_{i}Q=\frac{-h_{ti}}{h-\varepsilon_{0}}+\frac{h_{t}h_{i}}{(h-\varepsilon_{0})^{2}}.
\end{equation}
Then, with the assistance of \eqref{Up1}, at $x_{0}$, we also have
\begin{equation}
\begin{split}
\label{Up2}
0\geq\nabla_{ii}Q&=\frac{-h_{tii}}{h-\varepsilon_{0}}+\frac{2h_{ti}h_{i}+h_{t}h_{ii}}{(h-\varepsilon_{0})^{2}}-\frac{2h_{t}h^{2}_{i}}{(h-\varepsilon_{0})^{3}}\\
&=\frac{-h_{tii}}{h-\varepsilon_{0}}+\frac{h_{t}h_{ii}}{(h-\varepsilon_{0})^{2}}.
\end{split}
\end{equation}
\eqref{Up2} illustrates that
\begin{equation}
\begin{split}
\label{Up3}
-h_{tii}-h_{t}&\leq - \frac{h_{t}h_{ii}}{h-\varepsilon_{0}}-h_{t}\\
&=\frac{-h_{t}}{h-\varepsilon_{0}}[h_{ii}+(h-\varepsilon_{0})]\\
&=Q(b_{ii}-\varepsilon_{0}).
\end{split}
\end{equation}
 Furthermore,  there is
\begin{equation}
\begin{split}
\label{Up4*}
\partial_{t}Q&=\frac{-h_{tt}}{h-\varepsilon_{0}}+\frac{h^{2}_{t}}{(h-\varepsilon_{0})^{2}}\\
&=\frac{\omega_{n}(n+q-1)f}{2q(h-\varepsilon_{0})}\left[\frac{\partial [\det(\nabla ^{2}h+hI)]^{-1}}{\partial t}\frac{I_{q}(\Omega_{t})h^{p}}{\widetilde{V}_{q-1}(\Omega_{t},\overline{\nabla} h)\int_{\sn}fh^{p}dx}    \right.\\
&\left.\qquad \qquad   +\kappa I_{q}(\Omega_{t})\frac{d\widetilde{V}^{-1}_{q-1}(\Omega_{t},\overline{\nabla} h)}{d t}\frac{h^{p}}{\int_{\sn}fh^{p}dx}\right.\\
&\qquad \qquad \left. +\kappa I_{q}(\Omega_{t})\widetilde{V}^{-1}_{q-1}(\Omega_{t},\overline{\nabla}h)\left(-\frac{ph^{p}\int_{\sn}fh^{p-1}\frac{\partial h}{\partial t}dx}{(\int_{\sn}fh^{p}dx)^{2}}+\frac{ph^{p-1}\frac{\partial h}{\partial t}}{\int_{\sn}f h^{p}dx}  \right)\right].\\
&\qquad \qquad+Q+ Q^{2}.
\end{split}
\end{equation}
In the light of \eqref{Up4*}, by using \eqref{Up3}, at $x_{0}$, one has
\begin{equation}
\begin{split}
\label{Up5}
\frac{\partial [\det(\nabla^{2}h+hI)]^{-1}}{\partial t}&=-[\det(\nabla^{2}h+hI)]^{-2}\sum_{i}\frac{\partial[\det(\nabla^{2}h+hI)] }{\partial b_{ii}}(h_{tii}+h_{t})\\
&\leq[\det(\nabla^{2}h+hI)]^{-2}\sum_{i}\frac{\partial[\det(\nabla^{2}h+hI)] }{\partial b_{ii}}Q(b_{ii}-\varepsilon_{0})\\
&=\kappa Q[(n-1)-\varepsilon_{0}\sum_{i} b^{ii}].
\end{split}
\end{equation}
 Recall the fact that the eigenvalue of $\{b_{ij}\}$ and $\{b^{ij}\}$ are respectively the principal radii and the principal curvature of $\partial \Omega_{t}$ (see for example \cite{U91}). Thus we get
\begin{equation}
\begin{split}
\label{Up66}
\frac{\partial [\det(\nabla^{2}h+hI)]^{-1}}{\partial t}&=\kappa Q[(n-1)-\varepsilon_{0}H]\\
&\leq \kappa Q[(n-1)-\varepsilon_{0}(n-1)\kappa^{\frac{1}{n-1}}],
\end{split}
\end{equation}
where $H$ denotes the mean curvature of $\partial\Omega_{t}$, and the last inequality stems from  $H\geq (n-1)(\Pi_{i} b^{ii})^\frac{1}{n-1}=(n-1)\kappa^{\frac{1}{n-1}}$.

In addition, by virtue of \cite[Lemma 5.5]{HJ23}, at $x_{0}$, there is
\begin{equation*}
\begin{split}
\label{}
\frac{d }{d t}\widetilde{V}^{-1}_{q-1}(\Omega_{t},\overline{\nabla } h)&=-\frac{1}{\widetilde{V}^{2}_{q-1}(\Omega_{t},\overline{\nabla}h)}\frac{d}{dt}\widetilde{V}_{q-1}(\Omega_{t},\overline{\nabla}h)\\
&\leq \frac{|q-1|}{\widetilde{V}^{2}_{q-1}(\Omega_{t},\overline{\nabla}h)}Q(x_{0},t)\left(
    \int_{\Omega_t}
    \frac{dy}{|y-\delbar h|^{n+1-q}}
    +
    \frac{3C|q-1-n|}{n}
    \int_{\Omega_t}
    \frac{d y}{|y-\delbar h|^{n+2-q}}
  \right),
\end{split}
\end{equation*}
which together with \cite[Lemma 3.3]{HJ23} and Lemma \ref{C0}, for $q>2$, to yield that
\begin{equation} \label{Q33}
  \abs{\frac{d }{d t}\widetilde{V}^{-1}_{q-1}(\Omega_{t},\overline{\nabla } h)}
  \leq\frac{ C_1}{\widetilde{V}^{2}_{q-1}} Q(x_{0},t),
\end{equation}
where $C_1$ is a positive constant depending only on $n$, $q$, and the constant
$C$ is in Lemma \ref{C0}.

Substituting \eqref{Up66}, \eqref{Q33} into \eqref{Up4*}. Then, we have that at $x_{0}$,
\begin{equation}
\begin{split}
\label{finic2}
\partial_{t}Q
&\leq \frac{\omega_{n}f|n+q-1|}{2q(h-\varepsilon_{0})}\left[\kappa Q[(n-1)-\varepsilon_{0}(n-1)\kappa^{\frac{1}{n-1}}]I_{q}(\Omega_{t})\frac{h^{p}}{\widetilde{V}_{q-1}(\Omega_{t},\overline{\nabla } h)\int_{\sn}fh^{p}dx}\right.\\
&\quad \quad \left. +C_{1}\widetilde{V}^{-2}_{q-1}(\Omega_{t},\overline{\nabla} h)\kappa  \frac{I_{q}(\Omega_{t})h^{p}}{\int_{\sn}fh^{p}dx}Q(x_{0},t)\right.\\
&\quad \quad +\left. \kappa I_{q}(\Omega_{t})\widetilde{V}^{-1}_{q-1}(\Omega_{t},\overline{\nabla}h)\left( \frac{|p|h^{p}\int_{\sn}fh^{p-1}(h-\varepsilon_{0})Qdx}{(\int_{\sn}fh^{p}dx)^{2}}+\frac{|p|h^{p-1}(h-\varepsilon_{0})Q}{\int_{\sn}fh^{p}dx} \right) \right]\\
& \quad \quad +Q+ Q^{2}.\\
\end{split}
\end{equation}
Note that making use of the $C^{0}$ estimate and the definition of $\widetilde{V}_{q-1}$, for $q>1$, it suffices to have
\begin{equation}\label{qu1}
\frac{1}{C}\leq \widetilde{V}_{q-1}(\Omega_{t},\overline{\nabla} h)\leq C
\end{equation}
for a positive constant $C$, independent of $t$, depending only on $C^{0}$ estimate.

  Now, employing the priori estimates  into equation \eqref{AF1}, this allows us to assume $\kappa\approx Q>>1$. Then, using Lemma \ref{monJ1}, the $C^{0}, C^{1}$ estimates and substituting \eqref{qu1} into \eqref{finic2}, we conclude
\begin{equation}\label{Up13}
\partial_{t}Q\leq C_{0}Q^{2}(C_{1}-\varepsilon_{0}Q^{\frac{1}{n-1}})<0.
\end{equation}
 Hence, the ODE \eqref{Up13} tells that
\begin{equation}\label{Qbod}
Q(x_{0},t)\leq C
\end{equation}
for some $C>0$, independent of $t$.

Making again use of the priori estimates \eqref{AF1} and \eqref{Qbod}, for any $(x,t)$, we obtain
\begin{equation}
\begin{split}
\label{UpFal2}
\kappa&=\frac{(h-\varepsilon_{0})Q(x,t)+h}{f\frac{\omega_{n}}{2q}(n+q-1)I_{q}(\Omega_{t})\frac{h^{p}}{\widetilde{V}_{q-1}(\Omega_{t},\overline{\nabla } h)\int_{\sn}fh^{p}dx}}\\
&\leq \frac{(h-\varepsilon_{0})Q(x_{0},t)+h}{f\frac{\omega_{n}}{2q}(n+q-1)I_{q}(\Omega_{t})\frac{h^{p}}{\widetilde{V}_{q-1}(\Omega_{t},\overline{\nabla } h)\int_{\sn}fh^{p}dx}}\leq C
\end{split}
\end{equation}
for some $C> 0$, independent of $t$. So, the upper bound of Gauss curvature is established.
\end{proof}
Now, our aim is to gain the lower bound of \eqref{PUL}.  Taking advantage of the auxiliary function
\begin{equation}\label{LFun}
E(x,t)=\log (\sum\lambda (\{b_{ij}\}))-\tilde{d}\log h(x,t)+\tilde{l}|\nabla  h|^{2},
\end{equation}
where $\tilde{d}$ and $\tilde{l}$ are positive constants to be specified later, $\sum\lambda(\{b_{ij}\})$ is the sum of eigenvalue of $\{b_{ij}\}$. For any fixed $t\in (0,+\infty)$, assume that the maximum of $E(x,t)$ is attained at $x_{0}$ on $\sn$. Now we choose a local orthonormal frame such that $\{b_{ij}\}(x_{0},t)$ is diagonal, then \eqref{LFun} is transformed into
\begin{equation}\label{LFun2}
\widetilde{E}(x,t)={\rm log}(\sum_{j}b_{jj})-\tilde{d}\log h(x,t)+\tilde{l}|\nabla h|^{2}.
\end{equation}
Utilizing again the above assumption, thus, for any fixed $t\in(0,\infty)$, $\widetilde{E}(x,t)$ has a local maximum at $x_{0}$, which implies that, at $x_{0}$, it yields
\begin{equation}
\begin{split}
\label{gaslw1}
0=\nabla_{i}\widetilde{E}&=\frac{1}{\sum_{j}b_{jj}}\sum_{j}\nabla_{i}b_{jj}-\tilde{d}\frac{h_{i}}{h}+2\tilde{l} \sum_{j}h_{j}h_{ji}\\
&=\frac{1}{\sum_{j}b_{jj}}\sum_{j}\nabla_{i}b_{jj}-\tilde{d}\frac{h_{i}}{h}+2\tilde{l}h_{i}h_{ii},
\end{split}
\end{equation}
and
\begin{equation}\label{gaslw2}
0\geq \nabla_{ii}\widetilde{E}=\frac{1}{\sum_{j}b_{jj}}\sum_{j}\nabla_{ii}b_{jj}-\frac{1}{(\sum_{j}b_{jj})^{2}}(\sum_{j}\nabla_{i}b_{jj})^{2}-\tilde{d}\left(\frac{h_{ii}}{h}-\frac{h^{2}_{i}}{h^{2}}\right)
+2\tilde{l}\left[\sum_{j}h_{j}h_{jii}+h^{2}_{ii}\right].
\end{equation}
Moreover, there is
\begin{equation}
\label{gaslw3}
\partial_{t}\widetilde{E}=\frac{1}{\sum_{j}b_{jj}}\sum_{j}\partial_{t}b_{jj}-\tilde{d}\frac{h_{t}}{h}+2\tilde{l} \sum_{j}h_{j}h_{jt}.
\end{equation}
On the other hand,
\begin{equation}
\begin{split}
\label{gaslw4}
\log(h-h_{t})&=\log\left(f\kappa \eta(t)\frac{h^{p}}{\widetilde{V}_{q-1}(\Omega_{t},\overline{\nabla} h)\int_{\sn}fh^{p}xdx}\right)\\
&=-{\rm log}\det(\nabla^{2}h+hI)+\psi(x,t),
\end{split}
\end{equation}
where
\begin{equation}\label{gaslw5}
\psi(x,t):=\log\left(f \eta(t)\frac{h^{p}}{\widetilde{V}_{q-1}(\Omega_{t},\overline{\nabla} h)\int_{\sn}fh^{p}dx}\right).
\end{equation}
Now, taking the covariant derivative of \eqref{gaslw4} with respect to $e_{j}$, it follows that
\begin{align}\label{gaslw6}
\frac{h_{j}-h_{jt}}{h-h_{t}}&=-\sum_{i,k} b^{ik}\nabla_{j}b_{ik}+\nabla_{j}\psi\notag\\
&=-\sum_{i} b^{ii}(h_{jii}+h_{i}\delta_{ij})+\nabla_{j}\psi,
\end{align}
and
\begin{equation}
\begin{split}
\label{gaslw7}
&\frac{h_{jj}-h_{jjt}}{h-h_{t}}-\frac{(h_{j}-h_{jt})^{2}}{(h-h_{t})^{2}}\\
&=-\sum_{i} b^{ii}\nabla_{jj}b_{ii}+\sum_{i,k} b^{ii}b^{kk}(\nabla_{j}b_{ik})^{2}+\nabla_{jj}\psi.
\end{split}
\end{equation}
On the other hand, the Ricci identity on sphere reads
\begin{equation*}
\nabla_{kk}b_{ij}=\nabla_{ij}b_{kk}-\delta_{ij}b_{kk}+\delta_{kk}b_{ij}-\delta_{ik}b_{jk}+\delta_{jk}b_{ik}.
\end{equation*}
This together with  \eqref{gaslw2}, \eqref{gaslw3},  and \eqref{gaslw7}, by a direct computation, at $x_{0}$, we have
\begin{equation}
\begin{split}
\label{gaslw82}
\frac{\partial_{t}\widetilde{E}}{h-h_{t}}
&=\frac{1}{\sum_{j}b_{jj}}\left[\sum_{j}\frac{(h_{jjt}-h_{jj}+h_{jj}+(n-1)h-(n-1)h+(n-1)h_{t})}{h-h_{t}}\right]\\
&\quad-\tilde{d}\frac{1}{h}\frac{h_{t}-h+h}{(h-h_{t})}+2\tilde{l}\frac{\sum_{j} h_{j}h_{jt}}{h-h_{t}}\\
&=\frac{1}{\sum_{j}b_{jj}}\left[-\sum_{j}\frac{(h_{j}-h_{jt})^{2}}{(h-h_{t})^{2}}+\sum_{i,j} b^{ii}\nabla_{jj}b_{ii}-\sum_{i,j,k} b^{ii}b^{kk}(\nabla_{j}b_{ik})^{2}-\sum_{j}\nabla_{jj}\psi\right]\\
&\quad +\frac{1}{h-h_{t}}-(n-1)\frac{1}{\sum_{j}b_{jj}}+\frac{\tilde{d}}{h}-\frac{\tilde{d}}{h-h_{t}}+2\tilde{l}\frac{\sum_{j} h_{j}h_{jt}}{h-h_{t}}\\
&\leq \frac{1}{\sum_{j}b_{jj}}\left[\sum_{i,j} b^{ii}(\nabla_{ii}b_{jj}-b_{jj}+b_{ii})-\sum_{i,j,k} b^{ii}b^{kk}(\nabla_{j}b_{ik})^{2}\right]\\
&\quad+\frac{1}{h-h_{t}}(1-\tilde{d})-\frac{1}{\sum_{j} b_{jj}}\sum_{j}\nabla_{jj}\psi+\frac{\tilde{d}}{h}+2\tilde{l}\frac{\sum_{j} h_{j}h_{jt}}{h-h_{t}}\\
&\leq \sum_{i}b^{ii}\left[ \frac{1}{(\sum_{j}b_{jj})^{2}}(\sum_{j}\nabla_{i}b_{jj})^{2}+\tilde{d}\left(\frac{h_{ii}}{h}-\frac{h^{2}_{i}}{h^{2}}\right)-2\tilde{l}\left(\sum h_{j}h_{jii}+h^{2}_{ii}\right)\right]+\frac{(n-1)^{2}}{\sum_{j}b_{jj}}\\
&\quad-\frac{1}{\sum_{j}b_{jj}}\sum_{i,j,k} b^{ii}b^{kk}(\nabla_{j}b_{ik})^{2}-\frac{1}{\sum b_{jj}}\sum_{j}\nabla_{jj}\psi+\frac{\tilde{d}}{h}+2\tilde{l}\frac{\sum_{j} h_{j}h_{jt}}{h-h_{t}}+\frac{1}{h-h_{t}}(1-\tilde{d}).
\end{split}
\end{equation}
To proceed further, for each $i$, there is
\begin{equation}
\begin{split}
\label{AMI}
&\sum_{j}b_{jj}\sum_{j,k}b^{kk}(\nabla_{j}b_{ik})^{2}\\
&\geq (\sum_{j}b_{jj})\sum_{j}b^{jj}(\nabla_{i}b_{jj})^{2}\\
&\geq (\sum_{j}\sqrt{b_{jj}b^{jj}(\nabla_{i}b_{jj})^{2}})^{2}\\
&=(\sum_{j}|\nabla_{i}b_{jj}|)^{2}\\
&\geq (\sum_{j}\nabla_{i}b_{jj})^{2}.
\end{split}
\end{equation}
Using \eqref{AMI}, one see
\begin{equation}\label{AM2}
\sum_{i}b^{ii}\frac{1}{(\sum_{j}b_{jj})^{2}}(\sum_{j}\nabla_{i}b_{jj})^{2}-\frac{1}{\sum_{j}b_{jj}}\sum_{i,j,k}b^{ii}b^{kk}(\nabla_{j}b_{ik})^{2}\leq 0.
\end{equation}
Thus, applying \eqref{gaslw6} and \eqref{AM2}, \eqref{gaslw82} turns into
\begin{equation}
\begin{split}
\label{gaslw8}
\frac{\partial_{t}\widetilde{E}}{h-h_{t}}
&\leq\sum_{i} b^{ii}\tilde{d}\left(\frac{h_{ii}+h-h}{h}-\frac{h^{2}_{i}}{h^{2}} \right)-2\tilde{l}\sum_{i} b^{ii}h^{2}_{ii}+2\tilde{l}\sum_{j}h_{j}\left[-\sum_{i} b^{ii}h_{jii}+\frac{h_{jt}}{h-h_{t}}\right]\\
&\quad
-\frac{1}{\sum_{j} b_{jj}}(\sum_{j}\nabla_{jj}\psi)+\frac{\tilde{d}}{h}+\frac{1}{h-h_{t}}(1-\tilde{d})+\frac{(n-1)^{2}}{\sum_{j}b_{jj}}\\
&\leq -\tilde{d}\sum_{i} b^{ii}-2\tilde{l}\sum_{i} b^{ii}(b^{2}_{ii}-2b_{ii}h)-\frac{1}{\sum_{j} b_{jj}}(\sum_{j}\nabla_{jj}\psi)+\frac{n\tilde{d}}{h}+\frac{1}{h-h_{t}}(1-\tilde{d})\\
&\quad+2\tilde{l}\sum_{j}h_{j}\left[\frac{h_{j}}{h-h_{t}}+\sum_{i} b^{ii}h_{j}-\nabla_{j}\psi\right]+\frac{(n-1)^{2}}{\sum_{j}b_{jj}}\\
&\leq-\frac{1}{\sum_{j} b_{jj}}(\sum_{j}\nabla_{jj}\psi)-2\tilde{l}\sum_{j} h_{j}\nabla_{j}\psi +(2\tilde{l}|\nabla h|^{2}-\tilde{d})\sum_{i} b^{ii}-2\tilde{l}\sum_{i} b_{ii}\\
&\quad +\frac{2\tilde{l}|\nabla h|^{2}+1-\tilde{d}}{h-h_{t}}+4(n-1)h\tilde{l}+\frac{n\tilde{d}}{h}+\frac{(n-1)^{2}}{\sum_{j}b_{jj}},
\end{split}
\end{equation}
in which,
\begin{equation}\label{chi}
\nabla_{j}\psi=\frac{f_{j}}{f}-\frac{\nabla_{j}\widetilde{V}_{q-1}(\Omega_{t},\overline{\nabla} h)}{\widetilde{V}_{q-1}(\Omega_{t},\overline{\nabla} h)}+p\frac{h_{j}}{h},
\end{equation}
and
\begin{equation}\label{chi2}
\begin{split}
\nabla_{jj}\psi=\frac{ff_{jj}-f^{2}_{j}}{f^{2}}-\frac{\nabla_{jj}\widetilde{V}_{q-1}(\Omega_{t},\overline{\nabla } h)}{\widetilde{V}_{q-1}(\Omega_{t},\overline{\nabla} h)}+\frac{(\nabla_{j}\widetilde{V}_{q-1}(\Omega_{t},\overline{\nabla} h))^{2}}{(\widetilde{V}_{q-1}(\Omega_{t},\overline{\nabla }h))^{2}}+p\frac{hh_{jj}-h^{2}_{j}}{h^{2}}.
\end{split}
\end{equation}
 Combining \eqref{LFun2}, \eqref{chi} and \eqref{chi2}, this gives
\begin{equation}
\begin{split}
\label{gaslw9*}
&-2\tilde{l}\sum_{j}h_{j}\nabla_{j}\psi-\frac{1}{\sum_{j}b_{jj}}(\sum_{j}\nabla_{jj}\psi)\\
&=-2\tilde{l}\sum_{j} h_{j}\left(\frac{f_{j}}{f}-\frac{\nabla_{j}\widetilde{V}_{q-1}(\Omega_{t},\overline{\nabla } h)}{\widetilde{V}_{q-1}(\Omega_{t},\overline{\nabla} h)}+p\frac{h_{j}}{h}\right)\\
&\quad -\frac{1}{\sum_{j}b_{jj}}\sum_{j}\left(\frac{ff_{jj}-f^{2}_{j}}{f^{2}}-\frac{\nabla_{jj}\widetilde{V}_{q-1}(\Omega_{t},\overline{\nabla} h)}{\widetilde{V}_{q-1}(\Omega_{t},\overline{\nabla} h)}+\frac{(\nabla_{j}\widetilde{V}_{q-1}(\Omega_{t},\overline{\nabla} h))^{2}}{(\widetilde{V}_{q-1}(\Omega_{t},\overline{\nabla} h))^{2}}+p\frac{hh_{jj}-h^{2}_{j}}{h^{2}}\right).
\end{split}
\end{equation}
Employing \cite[Lemma 5.1, Lemma 5.2]{HJ23}, at $x_{0}$, we get
\begin{equation}\label{K11}
\nabla_{j}\widetilde{V}_{q-1}(\Omega_{t},\overline{\nabla}h)=\frac{(q-1)(n-q+1)}{n}\int_{\Omega_{t}}|y-z |^{q-n-3}[(y-z)\cdot e_{j}]b_{jj}dy,
\end{equation}
and
\begin{equation}
\begin{split}
\label{K22}
&\nabla^{2}_{jj}\widetilde{V}_{q-1}(\Omega_{t},\overline{\nabla}h)\\
&=\frac{(q-1)(n-q+1)}{n}\int_{\Omega_{t}}|y-z|^{q-n-3}\sum_{k}[(y-z)\cdot e_{k}]b_{jjk}dy\\
&\quad + \frac{(q-1)(n-q+1)}{n} \int_{\Omega_{t}}|y-z|^{q-n-3}\left[-(q-n-3)b^{2}_{jj}\frac{((y-z)\cdot e_{j})((y-z)\cdot e_{j})}{|y-z|^{2}}\right.\\
&\left. \qquad \quad \quad \quad \qquad \quad \qquad \quad \quad \quad \quad \qquad \qquad  -b^{2}_{jj} -b_{jj}[(y-z)\cdot x]\right]dy.
\end{split}
\end{equation}
Now, taking \eqref{K11} and \eqref{K22} into \eqref{gaslw9*}, at $x_{0}$, we derive
\begin{equation}
\begin{split}
\label{gaslw9*2}
&-2\tilde{l}\sum_{j}h_{j}\nabla_{j}\psi-\frac{1}{\sum_{j}b_{jj}}(\sum_{j}\nabla_{jj}\psi)\\
&=-2\tilde{l}\sum_{j} h_{j}\frac{f_{j}}{f}+\frac{1}{\widetilde{V}_{q-1}(\Omega_{t},\overline{\nabla} h)}\frac{(q-1)(n-q+1)}{n}\int_{\Omega_{t}}|y-z|^{q-n-3}\sum_{j}[(y-z)\cdot e_{j}]2\tilde{l}h_{j}b_{jj}dy\\
&\quad -2p\tilde{l}\frac{|\nabla h|^{2}}{h}-\frac{1}{\sum_{j}b_{jj}}\sum_{j}\left(\frac{ff_{jj}-f^{2}_{j}}{f^{2}}+\frac{(\nabla_{j}\widetilde{V}_{q-1}(\Omega_{t},\overline{\nabla} h))^{2}}{(\widetilde{V}_{q-1}(\Omega_{t},\overline{\nabla} h))^{2}}+p\frac{h(b_{jj}-h)-h^{2}_{j}}{h^{2}}\right)\\
&\quad +\frac{1}{\widetilde{V}_{q-1}(\Omega_{t},\overline{\nabla} h)}\left\{\frac{(q-1)(n-q+1)}{n}\int_{\Omega_{t}}|y-z|^{q-n-3}\sum_{k}[(y-z )\cdot e_{k}]\frac{1}{\sum_{j}b_{jj}}\sum_{j}b_{jjk}dy\right.\\
&\quad \quad \quad+ \left.\frac{(q-1)(n-q+1)}{n} \int_{\Omega_{t}}|y-z|^{q-n-3}\left[-(q-n-3)\frac{1}{\sum_{j}b_{jj}}\sum_{j}b^{2}_{jj}\frac{((y-z)\cdot e_{j})((y-z)\cdot e_{j})}{|y-z|^{2}}\right.\right.\\
&\left. \left. \qquad \quad \quad \quad \qquad \quad \qquad \quad \quad \quad \quad \qquad \qquad  -\frac{1}{\sum_{j}b_{jj}}\sum_{j}b^{2}_{jj} -[(y-z)\cdot x]\right]dy\right\}.
\end{split}
\end{equation}
Now, we set about estimating \eqref{gaslw9*2}. On the one hand, by the polar coordinate transform and $C^{0}$ estimate, for $q>3$, we can get
\begin{equation}\label{Q3}
\int_{\Omega_{t}}|y-z|^{q-n-3}dy\leq C
\end{equation}
for a positive constant $C$, independent of $t$. On the other hand, at $x_{0}$,  recall  \eqref{gaslw1}, i.e.,
 \begin{equation}\label{wii}
  \frac{1}{\sum_{j}b_{jj}}\sum_{j}b_{jjk}+ 2\tilde{l}h_{k}b_{kk}= \tilde{d}\frac{h_{k}}{h}+2\tilde{l}hh_{k}.
\end{equation}
 Substituting \eqref{qu1}, \eqref{Q3}, \eqref{wii},  $C^{0}$ and $C^{1}$ estimates into \eqref{gaslw9*2}. For $q> 3$, thus we obtain
\begin{equation}
\begin{split}
\label{DR}
&-2\tilde{l}\sum_{j}h_{j}\nabla_{j}\psi-\frac{1}{\sum_{j} b_{jj}}\sum_{j}\nabla_{jj}\psi\\
&\leq C_{1}\tilde{l}+C_{2}\sum_{j}b_{jj}+C_{3}\frac{1}{\sum_{j}b_{jj}}+C_{4}\tilde{d}+C_{5}.
\end{split}
\end{equation}
Now, choosing $\tilde{d}=2\tilde{l}\max_{\sn\times (0,+\infty)}|\nabla h|^{2}+1$. Applying \eqref{DR} into \eqref{gaslw8}, at $x_{0}$, we obtain
\begin{equation}\label{te}
\frac{\partial_{t}\widetilde{E}}{h-h_{t}}\leq \widetilde{C}_{1}\tilde{l}+C_{2}\sum_{j}b_{jj}+C_{3}\frac{1}{\sum_{j}b_{jj}}+C_{5} -2\tilde{l}\sum_{j}b_{jj}+4(n-1)h\tilde{l}+\frac{n\tilde{d}}{h}.
\end{equation}
If we take
\[\tilde{l}> C_{2}\]
and provide $\sum_{j}b_{jj}\gg 1$, there is
\begin{equation}\label{te2}
\frac{\partial_{t}\widetilde{E}}{h-h_{t}}\leq \widetilde{C}_{1}\tilde{l}+C_{3}\frac{1}{\sum_{j}b_{jj}}+C_{5}-\tilde{l}\sum_{j}b_{jj}+4(n-1)h\tilde{l}+\frac{n\tilde{d}}{h}< 0,
\end{equation}
which implies that
\begin{equation}\label{C1*}
E(x_{0},t)=\widetilde{E}(x_{0},t)\leq C
\end{equation}
for some $C> 0$, independent of $t$. Hence the proof of Theorem \ref{MMM} is completed.

\section{Proof of Theorem \ref{main}}
\label{Sec6}
In this section, our aim is to prove Theorem \ref{main}.

{ \bf \emph{ Proof of Theorem \ref{main}}.} From Sec. \ref{Sec4} and Sec. \ref{Sec5}, we conclude that $\eqref{GCF2}$ is uniformly parabolic in $C^{2}$ norm space. By virtue of \cite[Theorem 1.2]{HJ23}, one sees that $\widetilde{V}_{q}$ has the same smoothness as the boundary of convex body for $q>2$. Then, making use of the standard Krylov's regularity theory ~\cite{K87} of uniform parabolic equation, the estimates of higher derivatives of the solution to  \eqref{GCF2} can be naturally obtained, it illustrates that the long-time existence and regularity of the solution of \eqref{GCF2}. Furthermore, we have
\begin{equation}\label{ESM1}
||h||_{C^{i,j}_{x,t}(\sn\times[0,+\infty))}\leq C
\end{equation}
for each pairs of nonnegative integers $i$ and $j$, and for some $C>0$, independent of $t$.

Now, recall that Lemma \ref{monJ1}, it tells that
\begin{equation}\label{eq*}
\frac{d J(t)}{dt}\leq 0.
\end{equation}
Based on \eqref{eq*}, if there exists a $t_{0}$ such that
\begin{equation*}
\frac{d J(t)}{dt}\Big|_{t=t_{0}}= 0.
\end{equation*}
Then, it yields
\begin{equation}\label{Exil*}
h^{1-p}\det(h_{ij}+h\delta_{ij})\widetilde{V}_{q-1}(h,\overline{\nabla} h)=\frac{\omega_{n}}{2q}\frac{(n+q-1)I_{q}(\Omega)f(x)}{\int_{\sn}fh^{p}dx}, \quad  \forall x \in {\sn}.
\end{equation}
Let $\Omega=\Omega_{t_{0}}$, thus $\Omega$ is the desired solution.

On the other hand, if for any $t> 0$,
\begin{equation}\label{Phixiao}
\frac{d J(t)}{dt}<0.
\end{equation}

In the light of \eqref{ESM1}, by means of the Arzel\`a-Ascoli  theorem, we can extract a subsequence of $t$, denoted by $\{t_{k}\}_{k\in N}\subset (0,\infty)$, and there exists a smooth function $h(x)$ such that
\begin{equation*}
||h(x,t_{k})-h(x)||_{C^{i}({\sn})}\rightarrow 0
\end{equation*}
uniformly for any nonnegative integer $i$ as $t_{k}\rightarrow +\infty$. This implies that $h(x)$ is a support function. Let us denote by $\Omega$ the convex body determined by $h(x)$.  Thus, $\Omega$ is smooth, origin symmetric and strictly convex.

With the aid of  \eqref{ESM1} and the uniform estimates showed in Sec. \ref{Sec4} and Sec. \ref{Sec5}, we conclude that $J(t)$ is a bounded function in $t$ and $\frac{d J(t)}{dt}$ is uniformly continuous. So, for any $t>0$, using \eqref{Phixiao}, we get
\begin{equation*}
\int^{t}_{0}\left(-\frac{dJ(t)}{dt}\right)dt =J(0)-J(t)\leq C
\end{equation*}
for a positive constant $C$, which is independent of $t$. This gives
\begin{equation}\label{JH}
\int^{+\infty}_{0}\left(-\frac{d J(t)}{dt}\right)dt\leq C.
\end{equation}
\eqref{JH} implies that there exists a subsequence of time $t_{k}\rightarrow +\infty$ such that
\begin{equation}\label{}
\frac{dJ(t)}{dt}\Bigg|_{t=t_{k}}\rightarrow 0\quad as \ t_{k}\rightarrow +\infty.
\end{equation}
We are in a position to apply again the uniform estimates established in Sec. \ref{Sec4} and Sec. \ref{Sec5}. Clearly, there is a positive constant $\beta_{0}>0$ satisfying
\begin{equation}
\begin{split}
\label{ESM2}
\frac{d}{dt}J(t)\Bigg|_{t=t_{k}}&=-\int_{\sn}\frac{[-h+\frac{\kappa h^{p} \frac{\omega_{n}}{2q}(n+q-1)I_{q}(\Omega_{t})\frac{1}{\widetilde{V}_{q-1}(\Omega_{t},\overline{\nabla} h)}f}{\int_{\sn}f h^{p}dx}]^{2}}{\kappa h\frac{\omega_{n}}{2q}(n+q-1)I_{q}(\Omega_{t})\frac{1}{\widetilde{V}_{q-1}(\Omega_{t},\overline{\nabla} h)}}dx\Bigg|_{t=t_{k}}\\
&\leq -\beta_{0}\int_{\sn}[-h+\frac{\kappa h^{p} \frac{\omega_{n}}{2q}(n+q-1)I_{q}(\Omega_{t})\frac{1}{\widetilde{V}_{q-1}(\Omega_{t},\overline{\nabla} h)}f}{\int_{\sn}f h^{p}dx}]^{2}dx\Bigg|_{t=t_{k}}.
\end{split}
\end{equation}
Let $t_{k}\rightarrow +\infty$ in \eqref{ESM2}, we obtain
\begin{equation*}
0=\lim_{t_{k}\rightarrow +\infty}\frac{d J(t)}{dt}\Big|_{t=t_{k}}\leq -\beta_{0}\int_{\sn}[-h+\frac{\kappa h^{p} \frac{\omega_{n}}{2q}(n+q-1)I_{q}(\Omega)\frac{1}{\widetilde{V}_{q-1}(\Omega,\overline{\nabla} h)}f}{\int_{\sn}fh^{p}dx}]^{2}dx\leq 0,
\end{equation*}
which implies that
\begin{equation}\label{Exil*}
h^{1-p}\det(h_{ij}+h\delta_{ij})\widetilde{V}_{q-1}(h,\overline{\nabla } h)=\frac{\omega_{n}}{2q}\frac{(n+q-1)I_{q}(\Omega)f(x)}{\int_{\sn}f h^{p}dx}, \quad  \forall x \in {\sn}.
\end{equation}
Hence, $h(x)$ satisfies  \eqref{Exil*}. The proof is completed.

\end{document}